\documentclass{amsart}

\usepackage{amsfonts}
\usepackage{amssymb}
\usepackage{amsthm}
\usepackage{eucal}

\theoremstyle{plain} 
\newtheorem{theorem}{Theorem}[section]

\newtheorem{lemma}[theorem]{Lemma}
\newtheorem{corollary}[theorem]{Corollary}

\theoremstyle{definition}

\theoremstyle{remark}

\newcommand{\ol}{\overline}

\newcommand{\gen}[1]{\langle#1\rangle}

\newcommand{\vep}{\varepsilon}

\newcommand{\be}{\beta}

\begin{document}
\title{On classification of $(n+5)$-dimensional nilpotent $n$-Lie algebra of class two}

\author[Z. Hoseini]{Zahra Hoseini$^1$}
\author[F. Saeedi]{farshid saeedi$^1$}
\author[H. Darabi]{hamid darabi$^2$}

\date{}

\address{$^{1}$Department of Mathematics, Mashhad Branch, Islamic Azad University, Mashhad, Iran.}
\address{$^{2}$ Department of Mathematics, Esfarayen Branch, Islamic Azad University, Esfarayen, Iran.}

\email{zhosseini822@gmail.com}
\email{saeedi@mshdiau.ac.ir}
\email{darabi@iauesf.ac.ir}

\begin{abstract}
In this paper, we classify $(n+5)$-dimensional nilpotent $n$-Lie algebras of class two over the arbitrary field, when $n\ge 3$.
\end{abstract}

\maketitle
\section{\bf Introduction and Preliminaries}
The classification of low dimensional Lie algebras is one of the fundamental issues in Lie algebras Theory. The classification of Lie algebras can be found in many books and papers. In 1950 Morozov \cite{vm} proposed a classification of six-dimensional nilpotent Lie algebras over fields of characteristic 0. The classification of the 6-dimensional Lie algebras on the arbitrary field was shown by Cicalo et al.  \cite{sc.wg.cs}. Moreover, the 7-dimensional nilpotent Lie algebras over algebraically closed fields and real number field were classified in \cite{mg}. In 1985, Filippov \cite{vf} introduced the $n$-Lie algebras as a non-symmetrical linear vector space which satisfies the following Jacobian identity:
\[\left[ \left[ {{x}_{1}},{{x}_{2}},\ldots ,{{x}_{n}} \right],{{y}_{2}},\ldots ,{{y}_{n}} \right]=\sum\limits_{i=1}^{n}{\left[ {{x}_{1}},\ldots ,{{x}_{i-1}},\left[ {{x}_{i}},{{y}_{2}},\ldots ,{{y}_{n}} \right],{{x}_{i+1}},\ldots ,{{x}_{n}} \right]}\]
for all ${{x}_{i}},{{y}_{j}}\in L\,,\,1\le i\le n$, $2\le j\le n$. He also classified $n$-Lie algebras of dimensions $n$ and $n+1$ on the algebraically closed field with characteristic zero.

In 2008, Bai et al. \cite{rb.xw.wx}  classified $n$-Lie algebras of dimension $n+1$ on fields of characteristic two. Then, Bai \cite{rb.gs.yz} classified $n$-Lie algebras of dimension $n+2$ on the algebraically closed fields with characteristic zero.

Assume that ${{A}_{1}},\ldots \text{,}{{A}_{n}}$ are the subalgebra of $n$-Lie algebra $A$. Then, subalgebra $A$ generated by all vectors $\left[ {{x}_{1}},\ldots ,{{x}_{n}} \right]$$\left( {{x}_{i}}\in {{A}_{i}} \right)$ will be represented by the symbol $\left[ {{A}_{1}},\ldots ,{{A}_{n}} \right]$. The subalgebra ${{A}^{2}}=\left[ A,\ldots ,A \right]$ is called derived $n$-Lie algebra of $A$. The center of $n$-Lie algebra $A$ is defined as follows
\[Z\left( A \right)=\left\{ x\in A:\left[ x,A,\ldots ,A \right]=0 \right\}.\]
Assume ${{Z}_{0}}\left( A \right)=0$, then $i$th center of $A$ is defined inductively as ${{{Z}_{i}}\left( A \right)}/{{{Z}_{i-1}}\left( A \right)}\;=Z\left( {A}/{{{Z}_{i-1}}\left( A \right)}\; \right)$ for all $i\ge 1$. The notion of nilpotent $n$-Lie algebra was defined by Kasymov \cite{sk} as follows: We say that a $n$-Lie algebra $A$ is nilpotent if ${{A}^{s}}=0$ where, $s$ is a non-negative integer number. Note that ${{A}^{i}}$ is defined as induction by ${{A}^{1}}=A\,\,,\,{{A}^{i+1}}=\left[ {{A}^{i}},A,\ldots ,A \right]$. The $n$-Lie algebra $A$ is nilpotent of class $c$ if ${{A}^{c+1}}=0$, and ${{A}^{i}}\ne 0$, for each $i\le c$. 

An important category of $n$-Lie algebras of class 2, which plays an important role in nilpotent $n$-Lie algebras, is the Heisenberg $n$-Lie algebras. We call $n$-Lie algebra $A$, generalized Heisenberg of rank $k$, if ${{A}^{2}}=Z(A)$ and $\dim{{A}^{2}}=k$. 
In \cite{me.fs.hd1} the authors study the case when $k=1$, which is called later special Heisenberg $n$-Lie algebras.

The rest of our paper is organized as follows: Section 2 includes the results that are used frequently in the next section. In Section 3, we classify $(n+5)$-dimensional $n$-Lie algebras of class two.

\section{\bf Some Known Results}
In this section, we introduce some known and necessary results. 
\begin{theorem} \label{special Heisenberg}
(\cite{me.fs.hd1}) Every special Heisenberg $n$-Lie algebras have a dimension $mn+1$ for some natural number $m$, and it is isomorphic to
\[H\left( n,m \right)=\left\langle x,{{x}_{1}},\ldots ,{{x}_{nm}}:\left[ {{x}_{n\left( i-1 \right)+1}},{{x}_{n\left( i-1 \right)+2}},\ldots ,{{x}_{ni}} \right]=x,i=1,\ldots ,m \right\rangle .\]
\end{theorem}

\begin{theorem}\label{dimA^2=1}
(\cite{hd.fs.me}) Let $A$ be a $d$-dimensional nilpotent $n$-Lie algebra and $\dim{{A}^{2}}=1$. Then for some $m\ge 1$
\[A\cong H\left( n,m \right)\oplus F\left( d-mn-1 \right).\]
\end{theorem}

\begin{theorem}\label{A^2=Z(A)}
(\cite{hd.fs.me}) Let $A$ be a nilpotent $n$-Lie algebra of dimension $d=n+k$ for $3\le k\le n+1$ such that ${{A}^{2}}=Z\left( A \right)$ and $\dim{{A}^{2}}=2$. Then
\[A\cong \left\langle {{e}_{1}},\ldots ,{{e}_{n+k}}:\left[ {{e}_{k-1}},\ldots ,{{e}_{n+k-2}} \right]={{e}_{n+k}},\left[ {{e}_{1}},\ldots ,{{e}_{n}} \right]={{e}_{n+k-1}} \right\rangle .\]
This $n$-Lie algebra denoted by ${{A}_{n,d,k}}$.
\end{theorem}
For unification of notation, in what follows the $k$th $d$-dimmensional $n$-lie algebra will be denoted by ${{A}_{n,d,k}}$ 
\begin{theorem}\label{d<=n+2}
(\cite{hd.fs.me}) Let $A$ be a non-abelian nilpotent $n$-Lie algebra of dimension $d\le n+2$. Then
\[A\cong H\left( n,1 \right),H\left( n,1 \right)\oplus F\left( 1 \right),{{A}_{n,n+2,1}}\]
where

$A_{n,n+2,1}=\gen{e_1,\ldots,e_{n+2}:[e_1,\ldots,e_n]=e_{n+1},[e_2,\ldots,e_{n+1}]=e_{n+2}}$.

\end{theorem}

\begin{theorem}\label{6 dimensional nilpotent}
(\cite{sc.wg.cs}) 
\begin{itemize}
\item[(1)]
Over a field $F$ of characteristic different from $2$, the list of the isomorphisms types of $6$-dimensional nilpotent Lie algebras is the following: $L_{5,k}\oplus F$ with $k\in\{1,\ldots,9\}$ ; $L_{6,k}$ with $k\in\{10,\ldots,18,20,23,25,\ldots,28\}$; $L_{6,k}(\vep_1)$ with $k\in\{19,21\}$ and $\vep_1\in F^*/(\overset{*}{\mathop \sim})$; $L_{6,k}(\vep_2)$ with $k\in\{22,24\}$ and $\vep_2\in F/(\overset{*}{\mathop \sim})$.

\item[(2)]
Over a field $F$ of characteristic $2$, the isomorphism types of $6$-dimensional nilpotent Lie algebras are:
$L_{5,k}\oplus F$ with $k\in\{1,\ldots,9\}$; $L_{6,k}$ with $k\in\{10,\ldots,18,20,23,25,\ldots,28\}$; $L_{6,k}(\vep_1)$ with $k\in\{19,21\}$ and $\vep_1\in F^*/(\overset{*}{\mathop \sim})$; $L_{6,k}(\vep_2)$ with $k\in\{22,24\}$ and $\vep_2\in F/(\overset{*+}{\mathop \sim})$; $L_{6,k}^{(2)}$ with $k\in\{1,2,5,6\}$; $L_{6,k}^{(2)}(\vep_3)$ with $k\in\{3,4\}$ and $\vep_3\in F^*/(\overset{*+}{\mathop \sim})$; $L_{6,k}^{(2)}(\vep_4)$ with $k\in\{7,8\}$ and $\vep_4\in\{0,\omega\}$.

\end{itemize}

\end{theorem}
Eshrati et al. \cite{me.fs.hd2} classified $n$-Lie algebras  of $n+3$-dimensional for $n>2$. For the case $n=2$, we have 
\[H(2,2),H(2,1)\oplus F(2),L=<{{x}_{1}},...,{{x}_{5}}|[{{x}_{1}},{{x}_{2}}]={{x}_{4}},[{{x}_{1}},{{x}_{3}}]={{x}_{5}}>.\]
Additionally, for $n+4$-dimensional $n$-Lie algebras we have the following theorem:
\begin{theorem}\label{d=n+4}
(\cite{me.fs.hd2}) 
The only $(n+4)$-dimensional nilpotent $n$-Lie algebras of class two are:
$$ H(n,1)\oplus F(3), A_{n,n+4,1}, A_{n,n+4,2}, A_{n,n+4,3}, H(2,2)\oplus F(1), H(3,2), L_{6,22}^{{}}\left( \varepsilon  \right),\,L_{6,7}^{2}(\eta ).$$
\end{theorem}

\begin{theorem}\label{7 dimensional nilpotent}
(\cite{mg}) The $7$-dimensional nilpotent Lie algebras of class two over algebraically closed fields and real number field are:
$$H(2,1)\oplus F(4),H(2,2)\oplus F(2),H(2,3), L_{7,i},\quad 1 \leq
i \leq 10.$$
\end{theorem}

\section{\bf Classification of (n+5)-dimensional nilpotent $n$-Lie algebras of class two}
\vskip 0.4 true cm
In this section, we classify $(n+5)$-dimensional nilpotent $n$-Lie algebras of class two. The $n$-Lie algebra $A$ is nilpotent of class two, when $A$ is non-abelian and  ${{A}^{2}}\subseteq Z\left( A \right)$. 
The nilpotent $n$-Lie algebra of class two plays an essential role in some geometry problems such as the commutative Riemannian manifold. Additionally, the classification of nilpotent Lie algebras of class two is one of the important issues in Lie algebras.

We first prove a lemma for $3$-Lie algebras.
\begin{lemma}\label{3-Lie algebra of dimension 8}
Let $A$ be a 3-Lie algebra of dimension 8 such that ${{A}^{2}}=Z\left( A \right)$ and  $\dim{{A}^{2}}=2$. Then
\[A=\left\langle e_1,\ldots,e_8:[e_1,e_2,e_3]=e_7,[e_4,e_5,e_6]=e_8 \right\rangle .\] 
\end{lemma}
\begin{proof}
Let $A=\left\langle {{e}_{1}},\ldots ,{{e}_{8}} \right\rangle $ and ${{A}^{2}}=Z\left( A \right)=\left\langle {{e}_{7}},{{e}_{8}} \right\rangle $. We may assume that $\left[ {{e}_{4}},{{e}_{5}},{{e}_{6}} \right]={{e}_{8}}$. So, there are $\alpha _0,\beta _0, \alpha_{i,j,k}, \beta_{i,j,k}$ of $F$ such that 
\[\left\{ \begin{array}{lr}
 \left[ {{e}_{1}},{{e}_{2}},{{e}_{3}} \right]={{\alpha }_{0}}{{e}_{8}}+{{\beta }_{0}}{{e}_{7}}& \\ 
 \left[ {{e}_{i}},{{e}_{j}},{{e}_{k}} \right]={{\alpha }_{i,j,k}}{{e}_{8}}+{{\beta }_{i,j,k}}{{e}_{7}},&\text{}1\le i<j<k\le 6,\left( i,j,k \right)\ne \left( 1,2,3 \right),\left( 4,5,6 \right). \\
\end{array} \right.\]
Taking $I=\left\langle {{e}_{7}} \right\rangle $, thus $\dim{{\left( {A}/{I}\; \right)}^{2}}=1$. Therefore, by Theorem \ref{dimA^2=1}, ${A}/{I}\;$ is isomorphic to $H\left( 3,1 \right)\oplus F\left( 3 \right)$ or $H\left( 3,2 \right)\,$.

(i) Assume that ${A}/{I}\;\cong H\left( 3,1 \right)\oplus F\left( 3 \right)$. In this case, according to the structure of ${A}/{I}\;$, we have ${{\alpha }_{0}}={{\alpha }_{i,j,k}}=0$. Therefore, the brackets of $A$ are as follows:
\[\begin{cases}
[e_4,e_5,e_6]=e_8, \quad \quad [e_1,e_2,e_3]=\be_0e_7,&\\
[e_i,e_j,e_k]=\be_{i,j,k}e_7, \quad \quad 1\leq i<j<k\leq 6,(i,j,k)\ne (1,2,3),(4,5,6).
\end{cases}\]
Now, by choosing $J=\left\langle {{e}_{8}} \right\rangle $, and by taking into account $\dim{{\left( {A}/{J}\; \right)}^{2}}=1$, we have

$A/J\cong\gen{\ol{e}_1,\ldots,\ol{e}_7:[\ol{e}_1,\ol{e}_2,\ol{e}_3]=\be_0\ol{e}_7,[\ol{e}_i,\ol{e}_j,\ol{e}_k]=\be_{i,j,k}\ol{e}_7, 1\leq i<j<k\leq 6,(i,j,k)\ne (1,2,3),(4,5,6)}.$ According to the above brackets and special Heisenberg $n$-Lie algebra, the above algebra is isomorphic to $H\left( 3,1 \right)\oplus F\left( 3 \right)$. So, only one of  ${{\beta }_{0}}, {{\beta }_{i,j,k}}$ is equal to one and the other coefficients are zero.  

If one of the coefficients ${{\beta }_{i,j,k}}$  equal to one, then the condition ${{A}^{2}}=Z\left( A \right)$ will be false. So, we conclude $\beta_{i,j,k}=0$ for each $1\le i<j<k\le 6,\text{ }\left( i,j,k \right)\ne \left(1,2,3 \right),\left( 4,5,6 \right)$. Thus 
\[{{A}_{3,8,1}}=<{{e}_{1}},...,{{e}_{8}}|[{{e}_{1}},{{e}_{2}},{{e}_{3}}]={{e}_{8}},[{{e}_{4}},{{e}_{5}},{{e}_{6}}]={{e}_{7}}>.\]
(ii) Consider ${A}/{I}\;\cong H\left( 3,2 \right)$. According to the structure of ${A}/{I}\;$, we have ${{\alpha }_{0}}=1$ and ${{\alpha }_{i,j,k}}=0$. Therefore, the brackets of $A$:
\[\begin{cases}
[e_4,e_5,e_6]=e_8, \quad \quad [e_1,e_2,e_3]=e_8+\be_0e_7,&\\
[e_i,e_j,e_k]=\be_{i,j,k}e_7, \quad \quad 1\leq i<j<k\leq 6,(i,j,k)\ne (1,2,3),(4,5,6).
\end{cases}\]
Now, by choosing $J=\left\langle {{e}_{8}} \right\rangle $, we have  $\dim{{\left( {A}/{J}\; \right)}^{2}}=1$. Therefore, with respect to the structure of special Heisenberg $n$-Lie algebras, this algebra is isomorphic to   $H\left( 3,1 \right)\oplus F\left( 3 \right)$. So, only one of coefficients $ \beta_0, \beta _{i,j,k}$ is equal to one and the other coefficients are zero. Of course, if one of the coefficients ${{\beta }_{i,j,k}}$ is equal to one, we have a contradiction. So, we have ${{\beta }_{i,j,k}}=0$ for each $1\le i<j<k\le 6,\left( i,j,k \right)\ne \left( 1,2,3 \right),\left( 4,5,6 \right)$. Therefore, the brackets of $A$ are as follows: 
\[\left[ {{e}_{4}},{{e}_{5}},{{e}_{6}} \right]={{e}_{8}},\text{ }\left[ {{e}_{1}},{{e}_{2}},{{e}_{3}} \right]={{e}_{8}}+{{e}_{7}}.\]
By interchanging
$$e'_i=e_i,\quad1\leq i\leq 8, i \ne 7, \quad \quad e'_7=e_8+e_7$$
this algebra is isomorphic to ${{A}_{3,8,1}}$. so, the proof is completed.
\end{proof}
Now we are going to classify $(n+5)$-dimensional nilpotent $n$-Lie algebras of class two.

Assume that  $A$ is a $(n+5)$-dimensional nilpotent $n$-Lie algebras of class two where $n\ge 3$ and $A=\left\langle {{e}_{1}},\ldots ,{{e}_{n+5}} \right\rangle $ (see Theorem \ref{7 dimensional nilpotent} for the case $n=2$). If $\dim{{A}^{2}}=1$, then by Theorem \ref{dimA^2=1}, $A$ is isomorphic to one of the following algebras:
\[H\left( n,1 \right)\oplus F\left( 4 \right),\text{ }H\left( 3,2 \right)\oplus F\left( 1 \right),\text{ }H\left( 4,2 \right).\] 
Now, assume that $\dim{{A}^{2}}\ge 2$ and $\left\langle {e_{n+4}},{e_{n+5}} \right\rangle \subseteq {{A}^{2}}$. Therefore, ${A}/{\left\langle {{e}_{n+5}} \right\rangle }\;$  is a  $(n+4)$-dimensional nilpotent $n$-Lie algebras of class 2. It follows from Theorem \ref{d=n+4} that ${A}/{\left\langle {{e}_{n+5}} \right\rangle }\;$ is one of the following forms: 
\[H\left( n,1 \right)\oplus F\left( 3 \right),\text{ }{{A}_{n,n+4,1}},\text{ }{{A}_{n,n+4,2}},\text{ }{{A}_{n,n+4,3}},\text{ }H\left( 3,2 \right).\]

\medskip

Case 1: $A/\gen{e_{n+5}}\cong\gen{\ol{e}_1,\ldots,\ol{e}_{n+4}:[\ol{e}_1,\ldots,\ol{e}_n]=\ol{e}_{n+4}} \cong H(n,1) \oplus F(3)$. 

The brackets of $A$ are as follows:
\[\left\{ \begin{array}{ll}
 \left[ {{e}_{1}},{{e}_{2}},\ldots ,{{e}_{n}} \right]={{e}_{n+4}}+\alpha {{e}_{n+5}}& \\ 
 \left[ {{e}_{1}},\ldots ,{{{\hat{e}}}_{i}},\ldots ,{{e}_{n}},{{e}_{n+1}} \right]={{\alpha }_{i}}{{e}_{n+5}}&\text{ }1\le i\le n \\ 
 \left[ {{e}_{1}},\ldots ,{{{\hat{e}}}_{i}},\ldots ,{{e}_{n}},{{e}_{n+2}} \right]={{\beta }_{i}}{{e}_{n+5}},&\text{ }1\le i\le n \\ 
 \left[ {{e}_{1}},\ldots ,{{{\hat{e}}}_{i}},\ldots ,{{e}_{n}},{{e}_{n+3}} \right]={{\gamma }_{i}}{{e}_{n+5}},&\text{ }1\le i\le n \\ 
 \left[ {{e}_{1}},\ldots ,{{{\hat{e}}}_{i}},\ldots ,{{{\hat{e}}}_{j}},\ldots ,{{e}_{n}},{{e}_{n+1}},{{e}_{n+2}} \right]={{\theta }_{ij}}{{e}_{n+5}},&\text{ }1\le i<j\le n \\ 
 \left[ {{e}_{1}},\ldots ,{{{\hat{e}}}_{i}},\ldots ,{{{\hat{e}}}_{j}},\ldots ,{{e}_{n}},{{e}_{n+1}},{{e}_{n+3}} \right]={{\mu }_{ij}}{{e}_{n+5}},&\text{ }1\le i<j\le n \\ 
 \left[ {{e}_{1}},\ldots ,{{{\hat{e}}}_{i}},\ldots ,{{{\hat{e}}}_{j}},\ldots ,{{e}_{n}},{{e}_{n+2}},{{e}_{n+3}} \right]={{\lambda }_{ij}}{{e}_{n+5}},&\text{ }1\le i<j\le n \\ 
 \left[ {{e}_{1}},\ldots ,{{{\hat{e}}}_{i}},\ldots ,{{{\hat{e}}}_{j}},\ldots ,{{{\hat{e}}}_{k}},\ldots ,{{e}_{n}},{{e}_{n+1}},{{e}_{n+2}},{{e}_{n+3}} \right]={{\zeta }_{ijk}}{{e}_{n+5}},&\text{ }1\le i<j<k\le n \\ 
\end{array} \right..\]
By changing the base, we can result $\alpha =0$. Since $\dim{{A}^{2}}\ge 2$, we get  $\dim Z\left( A \right)\le 4$. 

First, we assume that  $\dim Z\left( A \right)=4$. In this case, without lose of generality, we can assume that $Z\left( A \right)=\left\langle {{e}_{n+2}},{{e}_{n+3}},{{e}_{n+4}},{{e}_{n+5}} \right\rangle $.  Consequently, the non-zero brackets of $A$ are as follows:

\[\left\{ \begin{array}{ll}
\left[ {{e}_{1}},{{e}_{2}},\ldots ,{{e}_{n}} \right]={{e}_{n+4}}& \\ 
\left[ {{e}_{1}},\ldots ,{{{\hat{e}}}_{i}},\ldots ,{{e}_{n}},{{e}_{n+1}} \right]={{\alpha }_{i}}{{e}_{n+5}},&\text{ }1\le i\le n \\ 
\end{array} \right..\]
At least one of the ${{\alpha }_{i}}$'s is non-zero. Without lose of generality, we can assume that ${{\alpha }_{1}}\ne 0$. By interchanging
\[e'_1=e_1+\sum_{i=2}^n(-1)^{i-1}\frac{\alpha_i}{\alpha_1}e_i,\quad e'_j=e_j,\quad2\leq j\leq n+4,\quad e'_{n+5}=\alpha_1e_{n+5}\]
we have
\[\left[ {{e}_{1}},\ldots ,{{e}_{n}} \right]={{e}_{n+4}},\text{ }\left[ {{e}_{2}},\ldots ,{{e}_{n+1}} \right]={{e}_{n+5}}.\]
We denote this algebra by ${{A}_{n,n+5,1}}$. Now, suppose $\dim Z\left( A \right)=3$.  Without lose of generality, we assume $Z\left( A \right)=\left\langle {{e}_{n+3}},{{e}_{n+4}},{{e}_{n+5}} \right\rangle $.  Therefore, the brackets of A are as follows:
\[\left\{ \begin{array}{ll}
 \left[ {{e}_{1}},{{e}_{2}},\ldots ,{{e}_{n}} \right]={{e}_{n+4}}& \\ 
 \left[ {{e}_{1}},\ldots ,{{{\hat{e}}}_{i}},\ldots ,{{e}_{n}},{{e}_{n+1}} \right]={{\alpha }_{i}}{{e}_{n+5}},&\text{ }1\le i\le n \\ 
 \left[ {{e}_{1}},\ldots ,{{{\hat{e}}}_{i}},\ldots ,{{e}_{n}},{{e}_{n+2}} \right]={{\beta }_{i}}{{e}_{n+5}},&\text{ }1\le i\le n \\ 
 \left[ {{e}_{1}},\ldots ,{{{\hat{e}}}_{i}},\ldots ,{{{\hat{e}}}_{j}},\ldots ,{{e}_{n}},{{e}_{n+1}},{{e}_{n+2}} \right]={{\theta }_{ij}}{{e}_{n+5}},&\text{ }1\le i<j\le n \\ 
\end{array} \right..\]
‌Since $\dim{{\left( {A}/{\left\langle {{e}_{n+4}} \right\rangle }\; \right)}^{2}}=1$, we have ${A}/{\left\langle {{e}_{n+4}} \right\rangle }\;\cong H\left( n,1 \right)\oplus F\left( 3 \right)$. According to the structure of $n$-Lie algebras, we conclude that one of the coefficients
$${{\theta }_{ij}}\left( 1\le i<j\le n \right),{{\beta }_{i}}\left( 1\le i\le n \right),{{\alpha }_{i}}\left( 1\le i\le n \right)$$ 
is equal to one, and the others are zero. According to $Z\left( A \right)=\left\langle {{e}_{n+3}},{{e}_{n+4}},{{e}_{n+5}} \right\rangle $,   the coefficients of ${{\theta }_{ij}}\left( 1\le i<j\le n \right),{{\beta }_{i}}\left( 1\le i\le n \right),{{\alpha }_{i}}\left( 1\le i\le n \right)$ can not be equal to one. Without loss of generality, assume that
$${{\theta }_{12}}=1,{{\theta }_{ij}}=0\left( 1\le i<j\le n,\left( i,j \right)\ne \left( 1,2 \right) \right),{{\beta }_{i}}=0\left( 1\le i\le n \right),{{\alpha }_{i}}=0\left( 1\le i\le n \right).$$
So, the brackets of $A$ are as follows:
\[\left[ {{e}_{1}},\ldots ,{{e}_{n}} \right]={{e}_{n+4}},\text{ }\left[ {{e}_{3}},\ldots ,{{e}_{n+2}} \right]={{e}_{n+5}}.\]
We denote this algebra by ${{A}_{n,n+5,2}}$.

Now, assume that $\dim Z\left( A \right)=2$. Therefore $Z\left( A \right)={{A}^{2}}$. In the case $n\ge 4$, by Theorem \ref{A^2=Z(A)} we reach the algebra  ${{A}_{n,n+5,3}}$.

In the case $n=3$, according to Lemma \ref{3-Lie algebra of dimension 8}, the desired algebra is
\[A=\left\langle {{e}_{1}},\ldots ,{{e}_{8}}:\left[ {{e}_{1}},{{e}_{2}},{{e}_{3}} \right]={{e}_{7}},\left[ {{e}_{4}},{{e}_{5}},{{e}_{6}} \right]={{e}_{8}} \right\rangle .\]
Which is denoted by ${{A}_{3,8,7}}$.
 
\medskip

Case 2: $A/\gen{e_{n+5}}\cong\gen{\ol{e}_1,\ldots,\ol{e}_{n+4}:[\ol{e}_1,\ldots,\ol{e}_n]=\ol{e}_{n+3},[\ol{e}_2,\ldots,\ol{e}_{n+1}]=\ol{e}_{n+4}}$. 

In this case  $\dim{{A}^{2}}=3$, so  ${{A}^{2}}=\left\langle {{e}_{n+3}},{{e}_{n+4}},{{e}_{n+5}} \right\rangle $. Thus $3\le \dim Z\left( A \right)\le 4$.  Hence the bracket of $A$ are as follows: 
\[\left\{ \begin{array}{ll}
 \left[ {{e}_{1}},{{e}_{2}},\ldots ,{{e}_{n}} \right]={{e}_{n+3}}+\alpha {{e}_{n+5}}& \\ 
 \left[ {{e}_{2}},\ldots ,{{e}_{n+1}} \right]={{e}_{n+4}}+\beta {{e}_{n+5}}& \\ 
 \left[ {{e}_{1}},\ldots ,{{{\hat{e}}}_{i}},\ldots ,{{e}_{n}},{{e}_{n+1}} \right]={{\alpha }_{i}}{{e}_{n+5}},&\text{ }2\le i\le n \\ 
 \left[ {{e}_{1}},\ldots ,{{{\hat{e}}}_{i}},\ldots ,{{e}_{n}},{{e}_{n+2}} \right]={{\beta }_{i}}{{e}_{n+5}},&\text{ }1\le i\le n \\ 
 \left[ {{e}_{1}},\ldots ,{{{\hat{e}}}_{i}},\ldots ,{{{\hat{e}}}_{j}},\ldots ,{{e}_{n}},{{e}_{n+1}},{{e}_{n+2}} \right]={{\theta }_{ij}}{{e}_{n+5}},&\text{ }1\le i<j\le n \\ 
\end{array} \right..\]
By changing the base, we can take $\alpha =\beta =0$.  First, we assume that $\dim Z\left( A \right)=4$. It follows that $Z\left( A \right)=\left\langle {{e}_{n+2}},{{e}_{n+3}},{{e}_{n+4}},{{e}_{n+5}} \right\rangle $. Therefore   
\[\left\{ \begin{array}{ll}
 \left[ {{e}_{1}},{{e}_{2}},\ldots ,{{e}_{n}} \right]={{e}_{n+3}}& \\ 
 \left[ {{e}_{2}},\ldots ,{{e}_{n+1}} \right]={{e}_{n+4}}& \\ 
 \left[ {{e}_{1}},\ldots ,{{{\hat{e}}}_{i}},\ldots ,{{e}_{n}},{{e}_{n+1}} \right]={{\alpha }_{i}}{{e}_{n+5}},&\text{ }2\le i\le n \\ 
\end{array} \right..\]
Due to the derived dimension, we conclude that at least one of the ${{\alpha }_{i}}\text{ }\!\!'\!\!\text{ s }\left( 2\le i\le n \right)$ is non-zero. 
Without loss of generality, we assume that ${{\alpha }_{2}}\ne 0$. Taking
\[e'_2=e_2+\sum_{i=3}^n(-1)^{i}\frac{\alpha_i}{\alpha_2}e_i,\quad e'_j=e_j,\quad j=1,3,\ldots,n+4,\quad e'_{n+5}=\alpha_2e_{n+5}\]
Hence, the non-zero brackets of this algebra are as follows:
\[\left[ {{e}_{1}},\ldots ,{{e}_{n}} \right]={{e}_{n+3}},\text{ }\left[ {{e}_{2}},\ldots ,{{e}_{n+1}} \right]={{e}_{n+4}},\text{ }\left[ {{e}_{1}},{{e}_{3}},\ldots ,{{e}_{n+1}} \right]={{e}_{n+5}}.\] 
We denote this algebra by ${{A}_{n,n+5,4}}$.

Now, assume that $\dim Z\left( A \right)=3$, thus ${{A}^{2}}=Z\left( A \right)=\left\langle {{e}_{n+3}},{{e}_{n+4}},{{e}_{n+5}} \right\rangle $. Hence  
\[\left\{ \begin{array}{ll}
 \left[ {{e}_{1}},{{e}_{2}},\ldots ,{{e}_{n}} \right]={{e}_{n+3}}& \\ 
 \left[ {{e}_{2}},\ldots ,{{e}_{n+1}} \right]={{e}_{n+4}}& \\ 
 \left[ {{e}_{1}},\ldots ,{{{\hat{e}}}_{i}},\ldots ,{{e}_{n}},{{e}_{n+1}} \right]={{\alpha }_{i}}{{e}_{n+5}},&\text{ }2\le i\le n \\ 
 \left[ {{e}_{1}},\ldots ,{{{\hat{e}}}_{i}},\ldots ,{{e}_{n}},{{e}_{n+2}} \right]={{\beta }_{i}}{{e}_{n+5}},&\text{ }1\le i\le n \\ 
 \left[ {{e}_{1}},\ldots ,{{{\hat{e}}}_{i}},\ldots ,{{{\hat{e}}}_{j}},\ldots ,{{e}_{n}},{{e}_{n+1}},{{e}_{n+2}} \right]={{\theta }_{ij}}{{e}_{n+5}},&\text{ }1\le i<j\le n \\ 
\end{array} \right..\]
‌Since $\dim{{\left( {A}/{\left\langle {{e}_{n+3}},{{e}_{n+4}} \right\rangle }\; \right)}^{2}}=1$, we have ${A}/{\left\langle {{e}_{n+3}},{{e}_{n+4}} \right\rangle }\;\cong H\left( n,1 \right)\oplus F\left( 2 \right)$. According to the structure of $n$-Lie algebras, we conclude that only one of the coefficients ${{\theta }_{ij}}\left( 1\le i<j\le n \right)$, ${{\beta }_{i}}\left( 1\le i\le n \right)$, ${{\alpha }_{i}}\left( 2\le i\le n \right)$ is equal to one and the other are zero.
On account of $Z\left( A \right)=\left\langle {{e}_{n+3}},{{e}_{n+4}},{{e}_{n+5}} \right\rangle $,  the coefficients    ${{\alpha }_{i}}\left( 2\le i\le n \right)$   
can not be equal to one. We have two cases.  

(i) Only one of the coefficients ${{\beta }_{i}}\left( 1\le i\le n \right)$ is equal to one and the others are zero.
Without loss of generality, we assume that ${{\beta }_{1}}=1$ and the others are zero. So the non-zero brackets of algebra are as follows:
\[\left[ {{e}_{1}},\ldots ,{{e}_{n}} \right]={{e}_{n+3}},\text{ }\left[ {{e}_{2}},\ldots ,{{e}_{n+1}} \right]={{e}_{n+4}},\text{ }\left[ {{e}_{2}},\ldots ,{{e}_{n}},{{e}_{n+2}} \right]={{e}_{n+5}}\]
We denote this algebra by ${{A}_{n,n+5,5}}$.

(ii) Only one of the coefficients ${{\theta }_{ij}}\left( 1\le i<j\le n \right)$ is equal one and the others are zero.

Without loss of generality, we assume that ${{\theta }_{12}}=1$, and the others are zero. So, the non-zero brackets of algebra are as follows:
\[\left[ {{e}_{1}},\ldots ,{{e}_{n}} \right]={{e}_{n+3}},\text{ }\left[ {{e}_{2}},\ldots ,{{e}_{n+1}} \right]={{e}_{n+4}},\text{ }\left[ {{e}_{3}},\ldots ,{{e}_{n+2}} \right]={{e}_{n+5}}.\]
We denote this algebra by$A_{n,n+5,6}$.

\medskip

Case 3: $A/\gen{e_{n+5}}\cong\gen{\ol{e}_1,\ldots,\ol{e}_{n+4}:[\ol{e}_1,\ldots,\ol{e}_n]=\ol{e}_{n+3},[\ol{e}_3,\ldots,\ol{e}_{n+2}]=\ol{e}_{n+4}}$. 

In this case $\dim{{A}^{2}}=3$, and so ${{A}^{2}}=Z\left( A \right)=\left\langle {{e}_{n+3}},{{e}_{n+4}},{{e}_{n+5}} \right\rangle $.  The brackets of $A$ are as follows:
\[\left\{ \begin{array}{ll}
 \left[ {{e}_{1}},{{e}_{2}},\ldots ,{{e}_{n}} \right]={{e}_{n+3}}+\alpha {{e}_{n+5}} &\\ 
 \left[ {{e}_{3}},\ldots ,{{e}_{n+2}} \right]={{e}_{n+4}}+\beta {{e}_{n+5}}& \\ 
 \left[ {{e}_{1}},\ldots ,{{{\hat{e}}}_{i}},\ldots ,{{e}_{n}},{{e}_{n+1}} \right]={{\alpha }_{i}}{{e}_{n+5}},&\text{ }1\le i\le n \\ 
 \left[ {{e}_{1}},\ldots ,{{{\hat{e}}}_{i}},\ldots ,{{e}_{n}},{{e}_{n+2}} \right]={{\beta }_{i}}{{e}_{n+5}},&\text{ }1\le i\le n \\ 
 \left[ {{e}_{1}},\ldots ,{{{\hat{e}}}_{i}},\ldots ,{{{\hat{e}}}_{j}},\ldots ,{{e}_{n}},{{e}_{n+1}},{{e}_{n+2}} \right]={{\theta }_{ij}}{{e}_{n+5}},&\text{ }1\le i<j\le n\text{ and }\left( i,j \right)\ne \left( 1,2 \right) \\ 
\end{array} \right.\]
By changing the base, we can take  $\alpha =\beta =0$.  So
\[\left\{ \begin{array}{ll}
 \left[ {{e}_{1}},{{e}_{2}},\ldots ,{{e}_{n}} \right]={{e}_{n+3}}& \\ 
 \left[ {{e}_{3}},\ldots ,{{e}_{n+2}} \right]={{e}_{n+4}}& \\ 
 \left[ {{e}_{1}},\ldots ,{{{\hat{e}}}_{i}},\ldots ,{{e}_{n}},{{e}_{n+1}} \right]={{\alpha }_{i}}{{e}_{n+5}},&\text{ }1\le i\le n \\ 
 \left[ {{e}_{1}},\ldots ,{{{\hat{e}}}_{i}},\ldots ,{{e}_{n}},{{e}_{n+2}} \right]={{\beta }_{i}}{{e}_{n+5}},&\text{ }1\le i\le n \\ 
 \left[ {{e}_{1}},\ldots ,{{{\hat{e}}}_{i}},\ldots ,{{{\hat{e}}}_{j}},\ldots ,{{e}_{n}},{{e}_{n+1}},{{e}_{n+2}} \right]={{\theta }_{ij}}{{e}_{n+5}},&\text{ }1\le i<j\le n\text{ and }\left( i,j \right)\ne \left( 1,2 \right) \\ 
\end{array} \right.\]
Since  $\dim{{\left( {A}/{\left\langle {{e}_{n+3}},{{e}_{n+4}} \right\rangle }\; \right)}^{2}}=1$, we get ${A}/{\left\langle {{e}_{n+3}},{{e}_{n+4}} \right\rangle }\;\cong H\left( n,1 \right)\oplus F\left( 2 \right)$. According to the structure of $n$-Lie algebras, we conclude that at least one of the coefficients ${{\theta }_{ij}}\left( 1\le i<j\le n,\left( i,j \right)\ne \left( 1,2 \right) \right)$, ${{\beta }_{i}}\left( 1\le i\le n \right)$, ${{\alpha }_{i}}\left( 1\le i\le n \right)$ is equal one and the others are zero.
Since $Z\left( A \right)=\left\langle {{e}_{n+3}},{{e}_{n+4}},{{e}_{n+5}} \right\rangle $, we have three cases. 

(i) Only one of the coefficients ${{\alpha }_{i}}\left( 1\le i\le n \right)$ is equal to one and the others are zero.  Without loss of generality, we assume that ${{\alpha }_{1}}=1$, and the others are zero. So, we have 
\[\left[ {{e}_{1}},\ldots ,{{e}_{n}} \right]={{e}_{n+3}},\text{ }\left[ {{e}_{3}},\ldots ,{{e}_{n+2}} \right]={{e}_{n+4}},\text{ }\left[ {{e}_{2}},\ldots ,{{e}_{n}},{{e}_{n+1}} \right]={{e}_{n+5}}.\]
One can check that, this algebra is isomorphic to  ${{A}_{n,n+5,6}}$.
 
(ii) Only one of the coefficients ${{\beta }_{i}}\left( 1\le i\le n \right)$ is equal to one and the others are zero.  Without loss of generality, we assume that ${{\beta }_{1}}=1$, and the others are zero. Hence, the non-zero bracket of algebra are 
\[\left[ {{e}_{1}},\ldots ,{{e}_{n}} \right]={{e}_{n+3}},\text{ }\left[ {{e}_{3}},\ldots ,{{e}_{n+2}} \right]={{e}_{n+4}},\text{ }\left[ {{e}_{2}},\ldots ,{{e}_{n}},{{e}_{n+2}} \right]={{e}_{n+5}}.\]
One can easily see that, this algebra is isomorphic to  ${{A}_{n,n+5,6}}$. 

(iii) Only one of the coefficients of ${{\theta }_{ij}}\left( 1\le i<j\le n,\left( i,j \right)\ne \left( 1,2 \right) \right)$ is equal to one and the others are zero.  Without loss of generality, we assume that ${{\theta }_{13}}=1$, and the others are zero. Hence, the non-zero bracket of algebra are 

\[\left[ {{e}_{1}},\ldots ,{{e}_{n}} \right]={{e}_{n+3}},\text{ }\left[ {{e}_{2}},\ldots ,{{e}_{n+1}} \right]={{e}_{n+4}},\text{ }\left[ {{e}_{2}},{{e}_{4}},\ldots ,{{e}_{n+2}} \right]={{e}_{n+5}}.\]
One can easily see that, this algebra is isomorphic to  ${{A}_{n,n+5,5}}$. 

\medskip

Case 4: $A/\gen{e_{n+5}}\cong\gen{\ol{e}_1,\ldots,\ol{e}_{n+4}:[\ol{e}_1,\ldots,\ol{e}_n]=\ol{e}_{n+1},[\ol{e}_2,\ldots,\ol{e}_n,\ol{e}_{n+2}]=\ol{e}_{n+3}}$. 

In this case $\dim{{A}^{2}}=4$, thus ${{A}^{2}}=Z\left( A \right)=\left\langle {{e}_{n+1}},{{e}_{n+3}},{{e}_{n+4}},{{e}_{n+5}} \right\rangle $.   The brackets of this algebra are as follows:
\[\left\{ \begin{array}{ll}
 \left[ {{e}_{1}},{{e}_{2}},\ldots ,{{e}_{n}} \right]={{e}_{n+1}}+\alpha {{e}_{n+5}}& \\ 
 \left[ {{e}_{2}},\ldots ,{{e}_{n}},{{e}_{n+2}} \right]={{e}_{n+3}}+\beta {{e}_{n+5}}& \\ 
 \left[ {{e}_{1}},{{e}_{3}},\ldots ,{{e}_{n}},{{e}_{n+2}} \right]={{e}_{n+4}}+\gamma {{e}_{n+5}}& \\ 
 \left[ {{e}_{1}},\ldots ,{{{\hat{e}}}_{i}},\ldots ,{{e}_{n}},{{e}_{n+2}} \right]={{\beta }_{i}}{{e}_{n+5}},&\text{ }3\le i\le n \\ 
\end{array} \right.\]
We use the change of base. In fact, by letting $\alpha =\beta =\gamma =0$, we get 
\[\left\{ \begin{array}{ll}
 \left[ {{e}_{1}},{{e}_{2}},\ldots ,{{e}_{n}} \right]={{e}_{n+1}}& \\ 
 \left[ {{e}_{2}},\ldots ,{{e}_{n}},{{e}_{n+2}} \right]={{e}_{n+3}}& \\ 
 \left[ {{e}_{1}},{{e}_{3}},\ldots ,{{e}_{n}},{{e}_{n+2}} \right]={{e}_{n+4}}& \\ 
 \left[ {{e}_{1}},\ldots ,{{{\hat{e}}}_{i}},\ldots ,{{e}_{n}},{{e}_{n+2}} \right]={{\beta }_{i}}{{e}_{n+5}},&\text{ }3\le i\le n \\ 
\end{array} \right..\]
Due to the derived dimension, we conclude that at least one of the ${{\beta }_{i}}\left( 3\le i\le n \right)$ is non-zero. 
Without loss of generality, we assume that ${{\beta }_{3}}\ne 0$.
Applying the following transformations
\[e'_3=e_3+\sum_{i=4}^n(-1)^{i-1}\frac{\be_i}{\be_3}e_i,\quad e'_j=e_j,\quad j=1,2,4,\ldots,n+4,\quad e'_{n+5}=\be_3e_{n+5}\]
the non-zero brackets of algebra are as follows
\[\begin{array}{ll}
 \left[ {{e}_{1}},{{e}_{2}},\ldots ,{{e}_{n}} \right]={{e}_{n+1}},&\text{ }\left[ {{e}_{2}},\ldots ,{{e}_{n}},{{e}_{n+2}} \right]={{e}_{n+3}} \\ 
 \left[ {{e}_{1}},{{e}_{3}},\ldots ,{{e}_{n}},{{e}_{n+2}} \right]={{e}_{n+4}},&\text{ }\left[ {{e}_{1}},{{e}_{2}},{{e}_{4}},\ldots ,{{e}_{n}},{{e}_{n+2}} \right]={{e}_{n+5}} .\\ 
\end{array}\]		
We denote this algebra by ${{A}_{n,n+5,7}}$.  

\medskip

Case 5: $A/\gen{e_{n+5}}\cong H(3,2) \cong\gen{\ol{e}_1,\ldots,\ol{e}_7:[\ol{e}_1,\ol{e}_2,\ol{e}_3]=[\ol{e}_4,\ol{e}_5,\ol{e}_6]=\ol{e}_7}$. 

In this case, according to the structure of algebra, we get ${{A}^{2}}=Z\left( A \right)=\left\langle {{e}_{7}},{{e}_{8}} \right\rangle $.
According to Lemma \ref{3-Lie algebra of dimension 8}, we have $A=\left\langle {{e}_{1}},\ldots ,{{e}_{8}}:[e_1,e_2,e_3]=e_7,[e_4,e_5,e_6]=e_8 \right\rangle $.  Therefore, this algebra does not satisfy our conditions.

\begin{theorem}
The $(n+5)$-dimensional nilpotent $n$-Lie algebras of class two, for $n>2$, over the arbitrary field are:
$$H(n,1)\oplus F(4), H(3,2)\oplus F(1),H(4,2),A_{3,8,7},A_{n,n+5,i}, \quad 1\leq i \leq 7.$$
\end{theorem}
According to the above theorem and Theorem \ref{7 dimensional nilpotent}, the main theorem of this paper is as follows:
\begin{corollary}
The $(n+5)$-dimensional nilpotent $n$-Lie algebras of class two are as following
\[\begin{aligned}
& H(n,1)\oplus F(4), H(2,2)\oplus F(2), H(2,3),H(3,2)\oplus F(1), H(4,2), A_{3,8,7},\\ 
& A_{n,n+5,i}, \quad 1\leq i \leq 6,\quad L_{7,i}, \quad 1 \leq i \leq 10.  \\ 
\end{aligned}\]

These algebras are valid for the case $n=2$ on the integers field and algebraically closed field, and $n>2$ on the arbitrary field.
\end{corollary}

In the following table we list the algebras which are presented in this article.

\begin{center}
Table I
\begin{tabular}{|c|c|}
\hline
Nilpotent $n$-Lie algebra&Non-zero multiplications\\
	\hline
		$A_{n,n+4,1}$ & $[e_1,\ldots,e_n]={{e}_{n+3}},[e_2,\ldots,e_{n+1}]=e_{n+4}$ \\
		\hline
		${{A}_{n,n+4,2}}$ & $[e_1,\ldots,e_n]={{e}_{n+3}},\left[ {{e}_{3}},\ldots ,{{e}_{n+2}} \right]={{e}_{n+4}}\text{ }\left( n\ge 3 \right)$ \\ 
		\hline
		$A_{n,n+4,3}$&$\begin{array}{c}[e_1,\ldots,e_n]=e_{n+1},[e_2,\ldots,e_n,e_{n+2}]=e_{n+3},\\{[e_1,e_3,\ldots,e_n,e_{n+2}]}=e_{n+4}\end{array}$\\
	
		\hline
		${{A}_{n,n+5,1}}$ & $[e_1,\ldots,e_n]=e_{n+4},[e_2,\ldots,e_{n+1}]=e_{n+5}$ \\ 
		\hline
		${{A}_{n,n+5,2}}$ & $[e_1,\ldots,e_n]={{e}_{n+4}},[e_3,\ldots,e_{n+2}]=e_{n+5}\text{ }(n\ge 3)$ \\
		\hline
		${{A}_{n,n+5,3}}$ &
		$\left[ {{e}_{1}},\ldots ,{{e}_{n}} \right]={{e}_{n+5}},\left[ {{e}_{4}},\ldots ,{{e}_{n+3}} \right]={{e}_{n+4}}\text{ }\left( n\ge 4 \right)$ \\
		\hline
		$A_{n,n+5,4}$&$\begin{array}{c}[e_1,\ldots,e_n]=e_{n+3},[e_2,\ldots,e_{n+1}]=e_{n+4},\\{[e_1,e_3,\ldots,e_{n+1}]}=e_{n+5}\end{array}$\\
		\hline
		$A_{n,n+5,5}$&$\begin{array}{c}[e_1,\ldots,e_n]=e_{n+3},[e_2,\ldots,e_{n+1}]=e_{n+4},\\{[e_2,\ldots,e_n,e_{n+2}]}=e_{n+5}\end{array}$\\
		\hline
		$A_{n,n+5,6}$&$\begin{array}{c}[e_1,\ldots,e_n]=e_{n+3},[e_2,\ldots,e_{n+1}]=e_{n+4},\\{[e_3,\ldots,e_{n+2}]}=e_{n+5}\end{array}$\\
		\hline
$A_{n,n+5,7}$&$\begin{array}{c}[e_1,\ldots,e_n]=e_{n+1},[e_1,e_2,e_4,\ldots,e_n,e_{n+2}]=e_{n+5},\\{[e_1,e_3,\ldots,e_n,e_{n+2}]}=e_{n+4},[e_2,\ldots,e_n,e_{n+2}]=e_{n+3} \end{array}$\\
\hline
		${{A}_{3,8,7}}$ & $[e_1,e_2,e_3]=e_7,[e_4,e_5,e_6]=e_8$ \\ 
		\hline
		${{L}_{7,1}}$ & $[e_1,e_2]=e_4,[e_1,e_3]=e_5$ \\ 
		\hline
		${{L}_{7,2}}$ & $[e_1,e_2]=e_4, [e_1,e_3]=e_5, [e_2,e_3]=e_6$ \\ 
		\hline
		${{L}_{7,3}}$ & $[e_1,e_2]=[e_3,e_4]=e_5, [e_3,e_4]=e_6$ \\ 
		\hline
		${{L}_{7,4}}$ & $[e_1,e_2]=e_5,\left[ {{e}_{3}},{{e}_{4}} \right]={{e}_{6}}$ \\ 
		\hline
        ${{L}_{7,5}}$ & $[e_1,e_2]=e_5, [e_2,e_3]=e_6, [e_2,e_4]=e_7$ \\
		\hline
		${{L}_{7,6}}$ & $[e_1,e_2]=e_5, [e_2,e_3]=e_6, [e_3,e_4]=e_7$ \\ 
		\hline
		$L_{7,7}$ & $[e_1,e_2]=[e_3,e_4]=e_5,[e_2,e_3]=e_6,[e_2,e_4]=e_7$ \\
		\hline
		$L_{7,8}$ & $[e_1,e_2]=[e_3,e_4]=e_5,[e_1,e_3]=e_6,[e_2,e_4]=e_7$ \\
		\hline
		${{L}_{7,9}}$ & $[e_1,e_5]=[e_3,e_4]=e_6, [e_2,e_5]=e_7$ \\
		\hline
		$L_{7,10}$ & $[e_1,e_2]=[e_3,e_4]=e_6,[e_1,e_5]=[e_2,e_3]=e_7$ \\
		\hline
\end{tabular}
\end{center}



\begin{thebibliography}{0}
\bibitem{rb.gs.yz}
R. Bai, G. Song, Y. Zhang, On classification of $n$-Lie algebras, {\it Front. Math. China}, {\bf6}(4) (2011), 581--606.

\bibitem{rb.xw.wx}
R. Bai, X.L. Wang, W.Y. Xiao, H.W. An, Structure of low dimensional $n$-Lie algebras over a field of characteristic 2, {\it Linear Algebra Appl}, {\bf428}(8-9) (2008), 1912--1920.

\bibitem{lb}
L.R. Bosko, On Schur multipliers of Lie algebras and groups of maximal class,  {\it Int. J. Algebr. Comput}, {\bf20}(06) (2010), 807--821.

\bibitem{sc.wg.cs}
S. Cicalo, W.A. de Graaf, C. Schneider, Six-dimensional nilpotent Lie algebras, {\it Linear Algebra Appl}, {\bf436}(1) (2012), 163--189.

\bibitem{hd.fs.me}
H. Darabi, F. Saeedi, M. Eshrati, A characterization of finite dimensional nilpotent Filippov algebras, {\it J. Geom. Phys}, {\bf101} (2016), 100--107.

\bibitem{me.fs.hd1}
M. Eshrati, F. Saeedi, H. Darabi, On the multiplier of nilpotent $n$-Lie algebras,  {\it J. Algebra}, {\bf450} (2016), 162--172.

\bibitem{me.fs.hd2}
M. Eshrati, F. Saeedi, H. Darabi, Low dimensional nilpotent $n$-Lie algebras, submitted.

\bibitem{vf}
V.T. Filippov, $n$-Lie algebras, {\it Sib. Math. J}, {\bf26}(6) (1985), 879--891.
\bibitem{mg}
M.P. Gong, Classification of nilpotent Lie algebras of dimension 7, Ph.D. thesis, University of Waterloo, Waterloo, Canada, 1998.

\bibitem{sk}
S.M. Kasymov, Theory of $n$-Lie algebras, {\it Algebra Logic}, {\bf26}(3) (1987), 155--166.

\bibitem{vm}
V.V. Morozov, Classification des algebres de Lie nilpotentes de dimension 6, {\it Izv. Vyssh. Ucheb. Zar}, {\bf4} (1958), 161--171.

\bibitem{mw}
M.P. Williams, Nilpotent $n$-Lie algebras, {\it Comm. Algebra}, {\bf37}(6) (2009), 1843--1849.

\end{thebibliography}
\end{document}